\documentclass[leqno,11pt]{amsart} 
\usepackage{amssymb, amscd, verbatim, amsmath, MnSymbol, stmaryrd, color}
\theoremstyle{plain} 
 \newtheorem{theorem}{\indent\sc Theorem}[section]

\newtheorem{thm}{\indent\sc Theorem}[section]
 \newtheorem{lem}[theorem]{\indent\sc Lemma}
 
 \newtheorem{prop}[theorem]{\indent\sc Proposition}

\theoremstyle{definition} 

\def\tr{\Delta}
\def\bea{\begin{eqnarray*}}
\def\eea{\end{eqnarray*}}
\def\be{\begin{eqnarray}}
\def\ee{\end{eqnarray}}
\def\a{\alpha}
\def\n{\nabla}

\def\d{\delta}
\def\o{\omega}
\def\ka{\kappa}

\numberwithin{equation}{section}

\begin{document}

 \title{$V$-static spaces with  positive isotropic curvature}

 \author{Gabjin Yun}
\address{Department of Mathematics\\  Myong Ji University\\
116 Myongji-ro Cheoin-gu\\ Yongin, Gyeonggi 17058, Republic of Korea. }
\email{gabjin@mju.ac.kr}

\author{Seungsu Hwang}
\address{Department of Mathematics\\ Chung-Ang University\\
84 HeukSeok-ro DongJak-gu \\ Seoul 06974, Republic of Korea.
}
\email{seungsu@cau.ac.kr}




 \maketitle

 \begin{abstract}
In this paper, we give a complete classification of critical metrics of the volume functional on a compact manifold $M$ with boundary $\partial M$ having  positive isotropic curvature.
We prove that for a  pair $(f, \kappa)$ of a nontrivial smooth function $f: M \to {\Bbb R}$ and 
a nonnegative real number $\kappa$,  if $(M, g)$ having  positive isotropic curvature satisfies
$$
Ddf - (\Delta f)g - f{\rm Ric} = \kappa g,
$$
then $(M, g)$ is isometric to a geodesic ball in ${\Bbb S}^n$ when $\kappa >0$, and 
either $M$  isometric to ${\Bbb S}^n_+$, or  the product
$I \times {\Bbb S}^{n-1}$, up to finite cover when $\kappa =0$.
 \end{abstract}
 
  \setlength{\baselineskip}{15pt}

\section{Introduction}

It is well-known that, on a compact manifold, a critical metric of 
the total scalar curvature functional restricted to the set of Riemannian metrics with fixed volume
 is Einstein.
Motivated by this as well as a result obtained in \cite{f-s-t}, 
Miao and Tam introduced in \cite{mt1} the volume functional and studied  variational properties of 
the volume functional on the space of constant scalar curvature metrics on a compact 
manifold with boundary.
In \cite{cem},  Corvino, Eichmair and Miao  considered Miao-Tam critical metrics in general context. 
In fact, they studied the modified problem of finding stationary points for the volume functional
on the space of metrics whose scalar curvature is equal to a given constant, and in this process, they
introduced notion of V-static metrics.

Let ${M}$ be an $n$-dimensional compact manifold  with smooth boundary $\partial M$. 
For a pair $(f, \kappa)$ of a smooth function $f: M\to {\Bbb R}$ and a real number $\kappa \in {\Bbb R}$, we say that $({M}, g,f,\kappa)$ is a  V-static space if $(M, g)$ satisfies the following
equation
\be 
s_g'^* f =\kappa \, g\quad  \mbox{on} \quad M. \label{cem1}
\ee 
Here, $s_g'^*$ is the $L^2$-adjoint operator of the linearized scalar curvature $s_g'$ with respect to the metric $g$.
The equation (\ref{cem1}) is called the $V$-static equation   and $f$ is  called a {\it V-static potential}.  
It is known  \cite{cem} that $g$ has constant scalar curvature, and
$f$ and $g$ are analytic in appropriate coordinates. If $f=0$ on $\partial M$ with $\kappa = 1$, 
a V-static metric reduces to a Miao-Tam critical metric. If $\kappa =0$, a V-static space becomes a vacuum static space.

Space forms and Ricci flat manifolds with $f=1$ and $\kappa=0$ are  V-static spaces. 
A scalar flat V-static space on a compact manifold is Ricci flat. A natural question is to classify V-static spaces. For example, 
an $n$-dimensional compact V-static space ($\kappa =1$) with parallel Ricci tensor and smooth boundary is isometric to a geodesic ball in a simply connected space form \cite{br1}. The same result holds when an $4$-dimensional V-static space is Bach flat \cite{bdr}. 

The main purpose of this paper is to classify  V-static spaces with positive isotropic curvature. The positive isotropic curvature (PIC in short) was first introduced by Micalleff and
 Moore \cite{mm88} in  consideration of the second variation of energy of maps from surfaces 
into $M$. We say that $(M,g)$ has PIC if and only if, for every orthonormal four-frame 
$\{e_1, e_2, e_3, e_4\}$, we have the inequality 
 $$ 
R_{1313}+R_{1414}+R_{2323}+R_{2424}-2R_{1234} >0.
$$
If $(M,g)$ has positive curvature operator, then it has PIC \cite{mm88}. Also, if the sectional
 curvature of $(M,g)$ is pointwise strictly quarter-pinced, then $(M,g)$ has PIC. 
The product metric on $S^{n-1}\times S^1$ has PIC. It is also known that the connected sum of
 manifolds with PIC  also admits a PIC metric, and a compact conformally flat $4$-manifold with
 positive scalar curvature  always has PIC. 
Note that PIC implies that $(M,g)$ has positive scalar curvature. 

It should be remarked that, when $\kappa =0$, a V-static space becomes a vacuum static space,
 and  there is a classification of closed vacuum static spaces with PIC. In fact, up to finite cover, 
a compact vacuum static space without boundary having  PIC is isometric to a sphere 
${\mathbb S}^n$, or the standard product of a circle ${\mathbb S}^1$ with 
an $(n-1)$-dimensional sphere ${\mathbb S}^{n-1}$ \cite{hy2}.

Our main result is the following.

 \begin{thm}\label{thm11} 
Let $(M, g,f, \ka)$ be a $V$-static space with $\partial M = f^{-1}(0)$ and $\kappa \ge 0$.  
Assume that $f >0$ on $M\setminus \partial M$ and  $(M, g)$ has PIC.
Then we have the following.
\begin{itemize}
\item[(1)] Assume that $\ka = 0$.
\begin{itemize}
\item[(i)] If  $\partial M$ is connected, then $M^n$ is isometric to ${\Bbb S}^n_+$.
\item[(ii)] If $\partial M$ is disconnected, then $\partial M$ has only two components, 
and, up to finite cover,  $M$ is isometric to $I \times {\Bbb S}^{n-1}$, the product of an interval
 with the standard $(n-1)$-sphere. 
\end{itemize}
\item[(2)] If $\ka >0$, then $M$ is isometric to a geodesic ball in ${\Bbb S}^n$.
\end{itemize}
\end{thm}
It should be noted that  there is a gap result when $\kappa <0$; if $\kappa <0$ and $f>0$ on $M\setminus \partial M$ with $\partial M=f^{-1}(0)$, then it follows from (\ref{lap1}) and the maximum principle that $\kappa$ should be bigger than $-\frac sna$ where $a=\max_Mf$. It would be interesting if one can find a rigidity result when $\kappa <0$.

One of main  ingredients in proving our result is to show that the maximum set of the potential function $f$ is totally geodesic if it is a hypersurface. Our proof for this property is quite  a long and looks some technical. This property shows that when $\kappa >0$, the maximum set is, in fact,
 a single point and the boundary $\partial M$ must be connected, which induces  $M$ is topologically  a disk.

\section{Preliminaries}
Let $(M,g,f, \kappa)$ be a V-static space with smooth boundary and a potential function $f$. 
Furthermore, we assume that $f>0 $ on $M\setminus \partial M$ and  $\partial M=f^{-1}(0)$ throughout the remainder of the paper.
The V-static equation (\ref{cem1}) can be written as 
\be
Ddf - (\Delta f)g - fr = k g,\label{eqn2017-1-13-1}
\ee
where $Ddf$ and $\Delta f$ denote the Hessian and Laplacian of $f$, and $r$  denotes
Ricci curvature  of the metric $g$. By taking the trace of (\ref{eqn2017-1-13-1}), we have
\be
\tr f = -\frac {n\kappa +s f}{n-1},\label{lap1}
\ee
where $s$ is the scalar curvature of $g$. Thus, we obtain
\be 
fz= Ddf +\frac {n\kappa+sf}{n(n-1)}\, g,\label{cem2}
\ee
where $z$ is the traceless Ricci tensor defined by $z = r - \frac{s}{n}g$.

\begin{prop} \label{prop2021-3-3-1}
Let $(M,g,f, \kappa)$ be a nontrivial V-static space with $\partial M = f^{-1}(0)$
and $f>0$ on $M\setminus \partial M$.
 Then, there are no critical points on the set $f^{-1}(0)$ if $\kappa \ge 0$.
 \label{prop1}
\end{prop}
\begin{proof}
In case $\kappa = 0$, a proof follows from \cite{f-m}. Suppose that there is a critical point of $f$ at $p\in f^{-1}(0)$. Let $\gamma$ be a unit-speed
geodesic starting at $p$ and define $h(t)=f(\gamma(t))$.
From (\ref{cem2}), we have
$$
 h''(t)= \left[z (\gamma'(t), \gamma'(t))  - \frac {s}{n(n-1)}\right] h(t) - \frac {\kappa}{n-1}
$$
with $h(0)=0$ and $h'(0)=0$. So, in case $\kappa = 0$, it follows from the uniqueness of ODE solution that $f$ vanishes identically, which is a contradiction.

Now assume $\kappa >0$ so that $f$ should be nontrivial. Since
$$
h''(0)= -\frac {\kappa}{n-1}<0, 
$$
 there exist points $q$ near $p$ such that $f(q)<0$, which is impossible.  
\end{proof}

An easy observation is that $|\n f|^2$ is constant on the boundary $\partial M$. In fact, if
$X$ is  tangent  to $\partial M$, by (\ref{cem2}), we have
$X(|\n f|^2) = 0$.

\begin{lem}\label{lem202-3-3-4}
Let $(M,g,f, \kappa)$ be a V-static space with smooth boundary $\partial M = f^{-1}(0)$ and $\kappa =  0$.
Then any connected component of $\partial M$ is totally geodesic hypersurface in $M$.
\end{lem}
\begin{proof}
We can take  $N = \frac{\n f}{|\n f|}$ as a unit normal vector field on (a component of)
 $\partial M$ by Proposition~\ref{prop2021-3-3-1}. 
Choosing a local frame $\{E_1, E_2, \cdots, E_{n-1}, N\}$ 
so that $\{E_1, E_2, \cdots, E_{n-1}\}$ are tangent to $\partial M$, we have
$E_i (|\n f|) = E_i \left(\frac{1}{|\n f|}\right) = 0$ on the set $\partial M$, which implies 
$D_{E_i}N = 0$.
\end{proof}

In case of $\kappa >0$, a connected component of $\partial M$ is not  totally geodesic. 
In fact, it is easy to compute that
$$
D_{E_i}N = - \frac{\kappa}{n-1}\frac{1}{|\n f|}E_i\quad
\mbox{and so}\quad \sum_{i=1}^{n-1}\langle D_{E_i}N, E_i\rangle = - \frac{\kappa}{|\n f|}.
$$

\vspace{.2in}

Now, to show various properties on $V$-statis spaces with PIC, we introduce a $3$-tensor $T$ 
 which plays an important role in investigating the structure of $V$-static spaces. 
For a $V$-static space $(M, g, f, \kappa)$ with boundary $\partial M = f^{-1}(0)$, we define $T$  by 
\be
T= \frac 1{n-2}df \wedge z +\frac 1{(n-1)(n-2)}i_{\nabla f}z \wedge g,
\ee
where $i_{\nabla f}$ denotes the usual interior product with respect to $\nabla f$. 
Here, $\phi \wedge \eta$ for a $1$-form $\phi$ and a symmetric 
$2$-tensor $b\in C^{\infty}(S^2M)$ is defined by
$$
(\phi \wedge b)(X,Y,Z)=\phi (X)b (Y,Z)-\phi (Y)b (X,Z).
$$
It follows from the definition of $T$ that the cyclic summation of $T_{ijk}$ is always vanishing, and
it is skew-symmetric in the fisrt two components.  

The differential $d^Db$ for a symmetric $2$-tensor $b$ is defined by
$$
d^Db(X,Y,Z)=D_Xb(Y,Z)-D_Yb(X,Z).
$$
Then, by taking the $d^D$ to (\ref{cem2}), we obtain
\bea 
df \wedge z(X,Y,Z)+f\, d^Dz(X,Y,Z)&=&d^DDdf (X,Y,Z)+ \frac s{n(n-1)}df\wedge g(X,Y,Z)\\
&=& R(X,Y,Z, \nabla f)+\frac s{n(n-1)}df\wedge g(X,Y,Z),
\eea
since the scalar curvature $s$ is constant. 
Thus,  from the curvature decomposition, we can obtain the following (cf. \cite{h-y})
\be
fC= \tilde{i}_{\nabla f}{\mathcal W}-(n-1)T, \label{eq119}
\ee
where  $C$ and ${\mathcal W}$ are   the Cotton tensor and Weyl tensor, respectively.
And $\tilde{i}$ is an interior product to the final factor defined by
    $$ 
\tilde{i}_{U}\xi (X,Y,Z)= \xi(X,Y, Z, U)
$$
for a $4$-tensor $\xi$ and a vector field $U$. 
For a $3$-tensor $\eta$, $\tilde{i}_{U}\eta$ is similarly defined. 
The Cotton tensor and Weyl tensor are related as follows \cite{Be}
$$
\d {\mathcal W} = - \frac{n-3}{n-2} C
$$
under the identification 
$$ 
\Gamma(T^*M\otimes \Lambda^2M) \equiv \Gamma(\Lambda^2M \otimes T^*M). 
$$
Moreover,  the cyclic summation of $C_{ijk}$ is vanishing as $T$,
and we have $C=d^Dz$ when the scalar curvature $s$ is constant.

Throughout this paper, we denote $N=\nabla f/|\nabla f|$ and $\alpha =z(N,N)$. 
It is clear that $\alpha$ is defined on $M$ except the critical points of $f$. However, 
since $|\alpha|\leq |z|$, the  $\alpha$ can be extended to all of $M$ as a $C^0$ function.

Another property on the tensor $T$ which will be used later is the following identity for the norm:
\be
|T|^2 = \frac{2}{(n-2)^2}|\n f|^2 \left(|z|^2 - \frac{n}{n-1}|i_N z|^2\right).
\label{eqn2021-3-8-2}
\ee

 The following property shows that vanishing of the tensor $T$ looks a little strong condition on 
 $V$-static spaces.
 We say a Riemannian manifold $(M, g)$ has harmonic curvature if $\d R = 0$ for the Riemannian curvature tensor $R$.

\begin{lem} Let $(M,g,f, \kappa)$ be a V-static space on a compact manifold $M$ with smooth boundary $\partial M$. If $T=0$, then $M$ has harmonic curvature. \label{lem23}
\end{lem}
\begin{proof} 
Since $\d R = - d^D r$, it suffices to prove that $C=0$. The idea of the proof follows from \cite{hy1}. By the definition of $T$, for an orthonormal frame $\{E_i\}_{1\leq i\leq n}$ with $E_n=N$, we have
\bea
(n-2)i_{\nabla f}T(E_i, E_j)&=& |\nabla f|^2z(E_i, E_j)+\frac 1{n-1}z(\nabla f, \nabla f)\d_{ij}\\
& & -\frac 1{n-1}z(\nabla f, E_i)df(E_j)-z(\nabla f, E_j)df(E_i).
\eea
Since $T=0$ by assumption, for $1\leq i, j\leq n-1$
\be
z_{ij}=z(E_i, E_j)= -\frac 1{n-1}\a \, \delta_{ij}. \label{eq1118}
\ee 
Also, for $1\leq i\leq n-1$ 
$$0=(n-2)i_{\nabla f}T(E_i, N)=  \frac {n-2}{n-1}\, z(E_i, N)|\nabla f|^2, $$
implying that
\be z(E_i, N)=0 \label{eq1119}\ee
for $1\leq i\leq n-1$.
By (\ref{eq1118}) and (\ref{eq1119}), we have
\be\langle i_{\nabla f}C, z\rangle = -\frac {\alpha}{n-1}\sum_{i=1}^{n-1} C(\nabla f, E_i, E_i)=\frac {\alpha}{n-1}C(\nabla f, N, N)=0.
\label{eq1120}
\ee

On the other hand, it follows from (\ref{eq119}) together with  $T=0$ that
\be 
f\, C =\tilde{i}_{\nabla f}{\mathcal W}, \label{eq1109}
\ee
and so $C(X, Y, \n f) = 0$. Since the cyclic summation of $C$ is vanishing, this implies
$$ 
C(Y, \nabla f, X)+C(\nabla f, X, Y) = 0.
$$
By taking the divergence $\d$ of (\ref{eq1109}), we can show the following (cf. \cite{hy1})
$$ 
-i_{\nabla f}C  +f\delta C =\delta (\tilde{i}_{\nabla f}{\mathcal W})= 
\frac {n-3}{n-2} i_{\n f} C +f \mathring{\mathcal W}z,
$$
where $\mathring{\mathcal W}z$ is defined by
$$
\mathring{\mathcal W}z(X, Y) = z({\mathcal W}(X, E_i)Y, E_i)
$$
for a local frame $\{E_i\}$. Thus,
\be
f\delta C=\frac {2n-5}{n-2} i_{\n f} C +f\mathring{\mathcal W}z. \label{eq1121}
\ee
It follows from the definition of $\mathring{\mathcal W}z$ together with (\ref{eq1109}) that
$$
\mathring{\mathcal W}z (\n f, X) = - f \langle i_XC, z\rangle
$$
for any vector field $X$. In particular, by (\ref{eq1120}), we have
\be 
\mathring{\mathcal W}z(\nabla f, \nabla f)=-f\langle i_{\nabla f}C , z\rangle = 0, \label{eq1122}
\ee
Consequently, 
\be \delta C(N,N)= \mathring{\mathcal W}z(N,N)= 0.\label{eqn2021-3-25-1}
\ee
Therefore, by (\ref{eq1120}), (\ref{eq1121}), and (\ref{eqn2021-3-25-1})
\bea
f\langle \delta C, z\rangle &=& \frac {2n-5}{n-2} \langle i_{\nabla f}C, z\rangle +f\langle \mathring{\mathcal W}z, z\rangle= f\sum_{1\leq i,j\leq n-1}\mathring{\mathcal W}z(E_i, E_j)z_{ij}\\
&=&-\frac {\a f}{n-1} \sum_{1\leq i\leq n-1}\mathring{\mathcal W}z(E_i, E_i)=\frac {\alpha f}{n-1}\mathring{\mathcal W}z(N,N)=0,
\eea
implying that  $$\langle \delta C, z\rangle =0.$$
Hence, from
$$ C(E_i, E_j, E_k)D_{E_i}z_{jk}= \frac 12 |C|^2,$$
we obtain
$$ 0=\int_M \langle \delta C, z\rangle =\int_{\partial M}\langle i_NC, z\rangle + \frac 12 \int_M |C|^2 
= \frac 12 \int_M |C|^2 .$$
This implies that $C=0$ on all of $M$, proving our Lemma.
\end{proof}

\section{V-static spaces with PIC}
In this section we assume that $(M,g,f, \kappa)$ is a V-static space on a compact manifold $M$
 with positive isotropic curvature. We claim that $z(\nabla f, X)=0$ for any vector field $X$ orthogonal to $\nabla f$ by following the idea in the case of Besse conjecture \cite{hy1}. First, we defined a $2$-form $\omega$ by
$$ 
\omega := df \wedge i_{\nabla f}z.
$$
By the definition of $T$, for any vector fields $X$ and $Y$
\be
\quad T(X,Y, \nabla f)=\frac 1{n-1}df\wedge i_{\nabla f}z(X,Y)=\frac 1{n-1}\omega (X,Y) =-\frac 1{n-1}f\, \tilde{i}_{\nabla f}C(X,Y).\label{eq1001}
\ee
Here, the last equality follows from (\ref{eq119}). Thus, we have
\be 
\omega = (n-1)\tilde{i}_{\nabla f}T =-f \,\tilde{i}_{\nabla f}C. \label{eqn2021-3-17-1}
\ee
Let $\{E_i\}_{i=1}^n$ be an orthonormal frame with $E_n=N=\nabla f/|\nabla f|$. It is clear that 
$$\omega (E_j, E_k)=0$$
for $1 \leq j,k\leq n-1$ by  definition of $\o$.
 Thus, if $\tilde{i}_{\nabla f}C(N, E_i)=0$ for $1\leq i\leq n-1$, then $\omega =0$.

Next, by replacing the function $1+f$ by $f$ in Lemmas 5.3 and 5.4 of \cite{hy1}, we can obtain
the following properties.

\begin{lem}\label{lem31}
As a $2$-form, $\tilde{i}_{\nabla f}C$ is an exact form. More precisely, we have
$$ \tilde{i}_{\nabla f}C=di_{\nabla f}z.$$
\end{lem}
\begin{lem}\label{lem32}
$\omega$ is a closed $2$-form, i.e., $d\omega =0$.
\end{lem}

On the other hand, let $\Omega=\{p\in M\,|\, \omega_p \neq 0$ on $T_pM\}$. Then $\Omega $ is an open subset of $M$. Note that $\Omega \cap \partial M = \emptyset$ by 
(\ref{eqn2021-3-17-1}).
Following an argument of \cite{hy1} by replacing $1+f$ by $f$, we have the following local inequality. The proof is not easy, and we refer \cite{hy1} for the proof.
\begin{lem} \label{lem33}
Suppose that $\omega_p \neq 0$  at $p\in M$. Then
$$ |D\omega|^2(p) \geq |\delta \omega|^2 (p).$$ 
\end{lem}

Now, we are ready to prove our claim. 

\begin{thm} \label{thm2021-3-17-2}
Let $(M,g, f, \kappa)$ be a V-static space on a compact manifold $M$ with smooth boundary 
$\partial M=f^{-1}(0)$. If $(M,g)$ has PIC, then the $2$-form $\omega$ vanishes on $M$. \label{thm34}
\end{thm}
\begin{proof} 
It suffices to prove that $\Omega =\emptyset$. Suppose, on the contrary, $\Omega \ne \emptyset$.
 For $p\in \Omega$, let $\Omega_0$ be a connected component of $\Omega$ containing $p$. 
Note that  $\tr \omega = -d\delta \omega$ by Lemma~\ref{lem32}.
By the Bochner-Weitzenb\"{o}ck formula for $2$-forms (c.f. \cite{pl1}, \cite{wu}),  we have
\be 
\frac 12 \tr |\omega|^2=\langle \tr \omega, \omega\rangle +|D\omega|^2+\langle E(\omega), \omega\rangle. \label{eq999}
\ee
Here, $E(\omega)$ is a local $2$-form containing isotropic curvature terms as its coefficients. In particular, if $(M,g)$ has PIC, then by Proposition 2.3 in \cite{pz} we have
\be \langle E(\omega), \omega \rangle >0.\label{eq997}\ee
Thus, by integrating (\ref{eq999})
\be
\frac 12 \int_{\Omega_0}\tr |\omega|^2 =\int_{\Omega_0}(\langle \tr \omega, \omega\rangle+|D\omega|^2)+\int_{\Omega_0} \langle E(\omega), \omega\rangle.
\label{eq998}
\ee
Since $\o = 0$ on $\partial \Omega_0$ and $\o$ is  a closed form, we have
$$
0=\int_{\Omega_0} |D\omega|^2-|\delta \omega|^2 +\langle E(\omega), \omega\rangle .
$$
However, by Lemma~\ref{lem33} and (\ref{eq997}),  the righthand side of  the above equation
 should be positive, which is a contradiction. Hence, we may conclude that $\Omega_0 =\emptyset$, 
or $\omega$ is trivial on $M$.
\end{proof}

\section{V-static spaces with $\omega=0$}

Let $(M,g,f, \kappa)$ be a V-static space with $\partial M= f^{-1}(0)$ and $\kappa \ge 0$ satisfying
$\omega =df\wedge i_{\n f}z =0$. Then, it is easy to see that 
\be
z(\nabla f, X)=0\label{eqn2019-5-27-2}
\ee
for any vector $X$ orthogonal to $\nabla f$ by  plugging $(\nabla f, X)$ into $\omega$. Thus, we may write 
$$
i_{\nabla f}z =\a df,\quad \mbox{where} \,\, \a = z(N, N), \,\, 
N = \frac{\n f}{|\n f|}
$$
as a $1$-form.

In this section we will prove that there are no critical points of $f$ except at maximum points of 
$f$ in $M$, and either the maximal set $f^{-1}(a)$ for $a=\max_{x\in M} f(x)$ is totally geodesic, or
contains only a single point. In particular,  if $\kappa >0$, we will show $f^{-1}(a)$ is a single point and the level set $f^{-1}(t)$ is homeomorphic to ${\Bbb S}^{n-1}$ including the boundary $\partial M$. 

First, we note that  $|\nabla f |$ is constant on each level set of $f$, since
$$ \frac 12 X(|\nabla f|^2)=f\, z(X, \nabla f) -\frac {n\kappa+ sf}{n(n-1)}\langle X, \nabla f\rangle =0 $$
for any tangent vector $X$ to a level set $f^{-1}(t), t>0$. Also, on $\partial M$ we have
$$ \nabla f(|\nabla f|^2)= -\frac {2\kappa}{n-1}|\nabla f|^2<0$$
by Proposition~\ref{prop1}.

\begin{lem}\label{lem41}
 Let $(M, g,f, \ka)$ be a $V$-static space with $\partial M = f^{-1}(0)$ and $\ka \ge  0$.
 Assume that $\o = 0$. Then there are no critical points of $f$ except at the maximum points of $f$.
\end{lem}

\begin{proof}
By the Bochner formula with (\ref{cem2}), we have
$$ \frac 12 \tr |\nabla f|^2 -\frac 1{2f}\nabla f(|\nabla f|^2) - \frac {\kappa}{(n-1)f} |\nabla f|^2 
=|Ddf|^2 \geq 0
$$
on the set $M\setminus \partial M$. Thus, by the maximum principle, there is no local maximum point of $|\nabla f|$ in the interior of $M$. 

Now, suppose that there is a critical point $p$ of $f$ with $f(p)=c<a=\max f$. Since $|\nabla f|$ is constant on each level set of $f$, $|\nabla f|=0$ on $f^{-1}(c)$. However, this implies that there should be a local maximum of $|\nabla f|^2$ in the set $\{x\in M\, \vert\, c <f(x) <a\}$, which is impossible by the above maximum principle.
\end{proof}

Lemma~\ref{lem41} shows that the maximal set $f^{-1}(a)$ is connected, and either
it is a single point or a hypersurface. In fact, if $f^{-1}(a)$ has at least two components, 
$f$ may have
 a critical point other than $f^{-1}(a)$. Furthermore, this property  shows that 
the boundary $\partial M$ has at most two connected components. In fact,  if it has more than
$3$ components, the potential function $f$ has a critical point other than $f^{-1}(a)$ or $f^{-1}(a)$
disconnected.
Lemma~\ref{lem41} also shows that if $\partial M$ is connected, then $f^{-1}(a)$ contains only a single point.

\begin{lem}\label{lem2021-3-3-2} 
Let $(M, g,f, \ka)$ be a $V$-static space  with 
 $\partial M = f^{-1}(0)$ and  $\ka \ge 0$ satisfying $\o = 0$. 
Suppose that $(M, g)$ has a positive scalar curvature. If
$\Sigma:= f^{-1}(a)$ is a hypersurface for $a = \max_M f$, then $\Sigma$ is totally geodesic.
\end{lem} 
\begin{proof}
By Lemma~\ref{lem41}, the boundary $\partial M$  has exactly two connected components.
First, we will show that $Ddf|_{\Sigma} = 0$.
For a sufficiently small $\epsilon >0$, $f^{-1}(a-\epsilon)$ has two connected components 
$\Sigma_{\epsilon}^+$ and $\Sigma_{\epsilon}^-$. 
Let $\nu$ be a unit normal vector field on $\Sigma=f^{-1}(a)$. On a tubular neighborhood of 
$\Sigma$, $\nu$ can be extended smoothly to a vector field $\tilde{\nu}$ such that 
$\tilde{\nu}\vert_{\Sigma}=\nu$ with 
$\tilde{\nu}\vert_{\Sigma_{\epsilon}^+} =N=\frac {\nabla f}{|\nabla f|}$ 
and $\tilde{\nu}\vert_{\Sigma_{\epsilon}^-}=-N=-\frac {\nabla f}{|\nabla f|}$.

The Laplacian of $f$ on $f^{-1}(a-\epsilon)=\Sigma_{\epsilon}^-\cup \Sigma_{\epsilon}^+$ is given by
\be 
\tr f =\tr ' f+Ddf(\tilde \nu, \tilde \nu)+m\langle \tilde \nu, \nabla f\rangle, \label{eq41}
\ee
where $\tr '$  and $m$ denote the intrinsic Laplacian and the mean curvature of $f^{-1}(a-\epsilon)$, respectively. Since $\nabla f=0$ at $\Sigma$ and $\tr 'f=0$ in (\ref{eq41}), 
by letting $\epsilon \to 0$, we have
\be 
\tr f = Ddf(\nu,\nu) \label{eq42}
\ee
on $\Sigma$. It is clear that $m$ is bounded on $\Sigma$. Thus, 
\bea
-\frac {n\kappa +sa}{n-1}=Ddf(\nu, \nu) =az(\nu, \nu)-\frac {n\kappa +sa}{n(n-1)},
\label{eqn2021-3-8-1}
\eea
implying that 
\be z(\nu, \nu)= -\left(\frac {\kappa}{a} +\frac{s}{n}\right) < 0.  \label{eq43}
\ee
Now let $p \in \Sigma$ be a point and choose an orthonormal basis $\{e_i\}_{i=1}^n$ at
 $p$ with $e_n=\nu$. Since $\Sigma$ is the maximum set of $f$, we have
$$ Ddf_p(e_i, e_i) = az_p(e_i, e_i)-\frac {n\kappa +sa}{n(n-1)}\leq 0
$$
for $1\leq i\leq n-1$, and so 
$$ z_p(e_i, e_i) \leq \frac {n\kappa +sa}{n(n-1)a} = -\frac 1{n-1}z_p(\nu, \nu).
$$
Since $\sum_{i=1}^{n-1}z_p(e_i, e_i)=-z_p(\nu, \nu)$, this  implies that 
\be 
z_p(e_i, e_i) =\frac {n\kappa+sa}{n(n-1)a} > 0\label{eq44}
\ee
and so
\be 
Ddf_p(e_i, e_i)=0\label{eq45}
\ee 
on $\Sigma$ for each $i,  1\leq i\leq n-1$. 

Second, we will show that $\Sigma$ is totally geodesic.
Since $z(\nu, X)=0$ for $X$ orthogonal to $\nu$ at $p\in \Sigma$, we may take the previously mentioned orthonormal basis $\{e_i\}_{i=1}^n$ so that $\{e_i\}_{i=1}^{n-1}$ diagonalize $z$ at $p$. 
Fix an $i (i=1,2,\cdots,n-1)$, say $i=1$, and  let $\gamma : [0,l)\to M$ be a unit speed geodesic 
such that $\gamma(0)=p$, $\gamma'(0)=e_1$ for some $l>0$. For $N(t):=N\circ \gamma (t)$, we have
\bea
D_{\gamma'} N&=&D_{\gamma'}\left( \frac {\nabla f}{|\nabla f|}\right)
= \gamma'\left( \frac 1{|\nabla f|}\right) \nabla f + \frac 1{|\nabla f|} Ddf(\gamma ' , \cdot)\\
&=&
\langle \gamma', N\rangle N \left( \frac 1{|\nabla f|}\right) \nabla f 
+ \frac 1{|\nabla f|} \left( fz(\gamma', \cdot)-\frac {n\kappa +sf}{n(n-1)} \gamma '\right).
\eea
Note that
$$
N\left(\frac 1{|\nabla f|}\right)=-\frac 1{|\nabla f|^2} \left( f\a -\frac {n\kappa +sf}{n(n-1)}\right).
$$
From the decomposition of $\gamma'$, we have
\be \gamma'= \langle \gamma', N\rangle N +\gamma'^{\top} = \gamma'^{\bot}+\gamma'^{\top}, \label{decomp}\ee
where $\gamma'^{\top}$ is the tangential component of $\gamma'$ to $f^{-1}(\gamma(t))$, we have
$$
|\nabla f| D_{\gamma'}N=  fz(\gamma'^{\top}, \cdot) - \frac {n\kappa +sf}{n(n-1)} \gamma'^{\top}.
$$
By taking the derivative in the direction $N$, we have
\bea
\lefteqn{Ddf(N,N)D_{\gamma'}N+|\nabla f|D_ND_{\gamma'}N}\\
&=& |\nabla f| z(\gamma'^{\top}, \cdot ) +f D_N(z(\gamma'^{\top}, \cdot))-\frac {s}{n(n-1)}|\nabla f| \gamma'^{\top}-\frac {n\kappa +sf}{n(n-1)}D_N \gamma'^{\top}.
\eea
Taking $t$ tends to $0+$, we have
\be 
-\frac {n\kappa +sa}{n-1} D_{e_1}\nu= a D_{\nu}(z(\gamma'^{\top}, \cdot))\bigg|_p
 -\frac {n\kappa +sa}{n(n-1)}D_{\nu} \gamma'^{\top}\bigg|_p. \label{eq47}
\ee

\vskip .5pc
\noindent
{\bf Claim:}  $\langle D_{\nu}(z(\gamma'^{\top}, \cdot)), e_1\rangle=0$ at the point $p$. 

\vskip .5pc
{\sc Proof of Claim.}
Let $\varphi (t)=f\circ \gamma(t)$. Note that $\varphi '(0)=df_p (\gamma'(0))=0$,  and by (\ref{eq45}) we have
$$
\varphi ''(0)=Ddf(\gamma'(0), \gamma'(0))=Ddf_p(e_1, e_1)=0.
$$
Since $\Sigma$ is a the maximum set of $f$,  $\varphi ''(t)\le 0$, or
$$ 
\varphi''(t) =Ddf(\gamma'(t), \gamma'(t))=\varphi(t) z(\gamma'(t), \gamma'(t))
-\frac {n\kappa +s\varphi(t)}{n(n-1)}\leq 0
$$
for a sufficiently small $t>0$. By (\ref{eq44}) $$ z(\gamma'(t), \gamma'(t))\geq 0$$ for a sufficiently small $t>0$, implying that
$$ 0\leq z(\gamma'(t), \gamma'(t))\leq \frac {n\kappa+s\varphi(t)}{n(n-1)\varphi(t)}.
$$
Defining  $\xi(t)=z(\gamma'(t), \gamma'(t))-  \frac {n\kappa+s\varphi(t)}{n(n-1)\varphi(t)}$, 
we have  $\xi(t)\leq 0$ with $\xi(0)=0$. Therefore,
$$ \xi'(0) =\frac d{dt}\bigg\vert_{t=0} z(\gamma'(t), \gamma'(t))  \leq 0.$$
Considering smooth extensions of $\gamma$ and $\xi$ to the interval $(-\bar{l}, 0]$, respectively, we have $ \frac d{dt}\vert_{t=0}z(\gamma'(t), \gamma'(t))  \geq 0$,
implying that 
$$ 
\frac d{dt}\bigg\vert_{t=0}z(\gamma'(t), \gamma'(t)) = 0.
$$
In particular, 
\be 
\nu_p (z(\gamma'(t), \gamma'(t)))\vert_{t=0}=0.\label{eq46}
\ee
 Let $\{E_k\}_{k=1}^n$ be an orthonormal frame extending $e_i$ near $p$ with $E_n=N$ and $E_1 =\frac {\gamma'^{\top}}{|\gamma'^{\top}|}$. Since $z(\gamma'^{\top}, \gamma'^{\bot})=0$, it follows from (\ref{decomp})  and (\ref{eq46}) that,  at $p$ we have
\bea
0&=& \nu_p[z(\gamma', \gamma')]= \nu_p[z(\gamma'^{\top}, \gamma'^{\top})+z(\gamma'^{\bot}, \gamma'^{\bot})]\\
&=& \nu_p[|\gamma'^{\top}|z(\gamma'^{\top}, E_i)]+\nu_p [ z(\gamma'^{\bot}, \gamma'^{\bot})]\\
&=& \nu_p(|\gamma'^{\top}|)z(\gamma'^{\top},E_i)+|\gamma'^{\top}|\nu_p[z(\gamma'^{\top},E_i)]+\nu_p [\langle \gamma', N\rangle^2 \a].
\eea
Since $|\gamma'^{\top}|\leq 1$ attains its maximum at $p$ and $\langle \gamma', N\rangle ^2$ attains its minimum $0$ at $p$, the first and the last term vanish. Thus, 
\be \nu_p[z(\gamma'^{\top},E_i)]=0.\label{eq49}\ee
On the other hand, since
\bea
D_N(z(\gamma'^{\top}, \cdot))&=&\sum_{j=1}^{n-1}D_N[z(\gamma'^{\top}, E_j)E_j]\\
&=& \sum_{j=1}^{n-1} N[z(\gamma'^{\top}, E_j)]E_j +\sum_{j=1}^{n-1} z(\gamma'^{\top}, E_j)D_NE_j,
\eea
we have
$$
\langle D_N[z(\gamma'^{\top}, \cdot)], E_1\rangle=N[z(\gamma'^{\top}, E_1)]+ \sum_{j=1}^{n-1} z(\gamma'^{\top}, E_j)\langle D_NE_j, E_1\rangle.
$$
As $t$ tends to $0$, we have
$$ \langle D_{\nu}[z(\gamma'^{\top}, \cdot)], e_1\rangle=\nu_p[z(\gamma'^{\top}, E_1)]+z(e_1, e_1)\langle D_{\nu}E_1, e_1\rangle =\nu_p[z(\gamma'^{\top}, E_1)] .
$$
Hence, our claim follows from (\ref{eq49}). \hfill$\Box$
\vskip .5pc
Therefore, by (\ref{eq47}) and {\bf  claim} above, 
$$
-\frac {n\kappa +sa}{n(n-1)}\, {\rm II}(e_1, e_1)=a \langle D_{\nu}(z(\gamma'^{\top}, \cdot)), e_1\rangle -\frac {n\kappa +sa}{n(n-1)}\langle D_{\nu} \gamma'^{\top}, e_1\rangle=0
$$
or ${\rm II}(e_1, e_1)=0$. Since $e_1$ is arbitrary, we may conclude that $\Sigma$ is totally geodesic.
\end{proof}

\vskip .5pc
Now, we consider the structure of the set of maximum points of $f$. We have the following result.

\begin{thm}\label{str001} 
Let $(M, g,f, \ka)$ be a $V$-static space on a compact manifold $M$ with 
boundary $\partial M = f^{-1}(0)$ and $\ka = 0$ satisfying $\o = 0$. Let $a = \max_M f$
and suppose $(M, g)$ has a positive scalar curvature.
Then we have the following.
\begin{itemize}
\item[(1)] if $\partial M$ is connected, then $ M$ is isometric to ${\Bbb S}^{n}_+$.
\item[(2)] if $\partial M$ is disconnected, then  $M$ is isometric to $I \times {\Bbb S}^{n-1}$,
the product of an interval with a standard sphere, up to finite cover.
\end{itemize}
\end{thm} 
\begin{proof}
 It is obvious that if $\partial M$ is connected, then
$f^{-1}(a)$ is a single point, and $\partial M$ is totally geodesic by Lemma~\ref{lem202-3-3-4}. 
Considering  two copies of $M$, and gluing the  boundaries, we can obtain
a compact smooth manifold $\hat M$ without boundary which is a vacuum static space and has PIC. 
Applying the main result in \cite{hy2}, up to finite cover, $\hat M$ is isometric to ${\Bbb S}^n$.
In particular, since $\hat M$ is Einstein and $\kappa = 0$, it is isometric to  ${\Bbb S}^n$ 
by Obata's result \cite{oba}. Thus, $(M, g)$ is isometric to ${\Bbb S}^{n}_+$
because the boundary $\partial M$ is totally geodesic.

Suppose $\partial M$ is disconnected so that $\Sigma:=f^{-1}(a)$ is a 
totally geodesic hypersurface in $M$, and $\partial M$ has exactly two connected components
$\Gamma_1$ and $\Gamma_2$. 
By Lemma~\ref{lem202-3-3-4}, each component $\Gamma_i (i=1,2)$ of $\partial M$ is totally geodesic. 
As in the case that $\partial M$ is connected, considering  two copies of $M$, and gluing the corresponding boundaries, we can obtain
a compact smooth manifold $\hat M$ without bounday which is a vacuum static space having PIC. 
By the main result in \cite{hy2}, up to finite cover, $\hat M$ is isometric to ${\Bbb S}^1 \times
{\Bbb S}^{n-1}$, and  so both $M$ is isometric to $I \times {\Bbb S}^{n-1}$.
the product of an interval with a standard sphere.
\end{proof}

\begin{thm}\label{str02} 
Let $(M, g,f, \ka)$ be a $V$-static space on a compact manifold $M$ with 
boundary $\partial M = f^{-1}(0)$ and $\ka > 0$ satisfying $\o = 0$. 
Let $a = \max_M f$ and suppose $(M, g)$ has a positive scalar curvature.
Then, $f^{-1}(a)$ contains only  a single point. In particular, $\partial M$ is connected, and
every level set $f^{-1}(t), 0< t< a$, is homeomorphic to a sphere ${\Bbb S}^{n-1}$ and $M$ is contractible.
\end{thm} 
\begin{proof}
Recall that the maximum set $\Sigma:= f^{-1}(a)$ is connected by Lemma~\ref{lem41}.
Assume $\kappa >0$ and $\Sigma$ is a hypersurface so that
$\Sigma$ is totally geodesic  by Lemma~\ref{lem2021-3-3-2}.
By (\ref{eq43}), we have 
$$
 r(\nu, \nu)= -\frac {\kappa}a <0
$$
on the set $\Sigma$. Therefore, on $\Sigma$ the stability operator becomes
$$
\int_{\Sigma}\left( |\nabla \phi|^2 -|{\rm II}|^2\phi^2 -r(\nu, \nu)\phi^2\right) \geq 0
$$
for any smooth function $\phi$ on $\Sigma$. Then, by Fredholm alternative (c.f. \cite{f-s}, Theorem 1) there exists a positive function $\varphi >0$ on $\Sigma$ such that 
$$\tr '\varphi +r(\nu, \nu)\varphi =0.$$
Here, $\tr '$ is the intrinsic Laplacian on $\Sigma$. However, since $\tr '\varphi =\frac {\kappa}a\varphi>0$ and $\Sigma$ is compact, by the maximum principle $\varphi $ should be constant, which is impossible. This contradiction implies that $\Sigma$ cannot be a hypersurface. 
Since $f^{-1}(a)$ contains only  a single point, by considering a small geodesic ball at $a$ and applying Lemma~\ref{lem41}, we can see that  every level set
$f^{-1}(t), 0< t< a$, is homeomorphic to a sphere ${\Bbb S}^{n-1}$. Also, by Isotopy lemma, $M$ is contractible.

\end{proof}

\section{Proof of Theorem~\ref{thm11}}
In this section we prove Theorem~\ref{thm11}. To prove Theorem~\ref{thm11}, it suffices to prove that $T=0$ due to Lemma~\ref{lem23}. 

Let $\bar{M}=M\setminus f^{-1}(a)$ with $a=\max_M f$. Fix a level set $f^{-1}(c)$ with $0 \le c < a$ and let $g'$ be the restriction of $g$ to $f^{-1}(c)$.
Consider a warped product metric $\bar{g}$ on $\bar{M}$ given by
\be 
\bar{g}= \frac {df}{|\nabla f|} \otimes \frac {df}{|\nabla f|} +|\nabla f|^2 g'.\label{eq51}
\ee
By Theorem~\ref{lem41}, $\bar{g}$ is smooth on $M$ except at the maximum point of $f$.  
As  Lemma 4.1 in \cite{hy1}, $\nabla f$ is a conformal Killing vector field with respect to $\bar{g}$. 

\begin{lem} \label{lem51} 
For  $(\bar{M},\bar{g})$, we have  
$$ \frac 12 {\mathcal L}_{\nabla f}\bar{g}=N(|\nabla f|)\bar{g}= \frac 1n (\bar{\tr}f)\bar{g}.$$
\end{lem}

\begin{lem}\label{lem2020-6-30-1}
Let $(M, g,f, \ka)$ be a $V$-static space  with $\partial M = f^{-1}(0)$ and $\ka \ge 0$ 
satisfying $\o = 0$. Let $a = \max_M f$. Suppose $(M, g)$ has a positive scalar curvature.
 Then level hypersurfaces given by $f$ are homothetic to each other except $f^{-1}(a)$.
\end{lem}
\begin{proof}
By Lemma~\ref{lem51} , we can choose a local coordinate system $(u^i)$ in a neighborhood of any hypersurface $f^{-1}(c)$, $0<c<a$  such that
\be
\bar g = (du^1 )^2 + |\n f|^2 \eta_{ij}(u^2, \cdots, u^n) du^i \otimes du^j, \label{eqn2020-1-17}
\ee
where $du^1 = \frac{df}{|\n f|}$ and
 the functions $\eta_{ij}$ depend only on $u^2, \cdots, u^n$ (cf. \cite{tas2} or \cite{tas}).
 Comparing this to (\ref{eq51}), we have
 $$
 g' = \eta_{ij}(u^2, \cdots, u^n) du^i \otimes du^j.
 $$
 These show that level hypersurfaces are homothetic to each other and to $f^{-1}(c)$ 
with $\eta_{ij}(u^2, \cdots, u^n) du^i \otimes du^j$ as metric form.
\end{proof}

 \begin{thm}\label{thm111} 
Let $(M, g,f, \ka)$ be a nontrivial $V$-static space with $\partial M = f^{-1}(0)$, $f>0$ on $M$ 
and $\kappa \ge 0$.  Suppose that $(M, g)$ has PIC.
Then we have the following.
\begin{itemize}
\item[(1)] Assume that $\kappa = 0$.
\begin{itemize}
\item[(i)] If  $\partial M$ is connected, then $M^n$ is isometric to ${\Bbb S}^n_+$.
\item[(i)] If $\partial M$ is disconnected, then $\partial M$ has only two components, and
 $M$ is isometric to the product $I \times {\Bbb S}^{n-1}$ of an interval with a standard sphere, 
up to finite cover. 
\end{itemize}
\item[(2)] If $\kappa >0$, then  $M$ is isometric to a geodesic ball in ${\Bbb S}^n$.
\end{itemize}
\end{thm}
\begin{proof}
Since $(M, g)$ has PIC, we have $\o:=df \wedge i_{\n f}z = 0$ by Theorem~\ref{thm2021-3-17-2}
and $(M, g)$ has a positive constant scalar curvature. In case $\kappa = 0$, the conclusion follows from Theorem~\ref{str001} directly.
 
Now, assume that $\kappa > 0$ so that $\partial M$ is connected and the maximum set $f^{-1}(a)$ with $a = \max_M f$ is a single point. In particular, every level set
$f^{-1}(t), 0< t< a$, is homeomorphic to a sphere ${\Bbb S}^{n-1}$ by Theorem~\ref{str02}.

By Lemma~\ref{lem2020-6-30-1},  the given metric $g$  can be written as 
$$ 
g= \frac {df}{|\nabla f|} \otimes \frac {df}{|\nabla f|} +b(f)^2g',
$$
where $b(f)>0$ is a positive function depending only on $f$ and $g'$ is a metric restricted to a hypersurface $f^{-1}(c)$, $0 < c<a$. As in \cite{hy1}, we may obtain 
$$ 
\frac 12 {\mathcal L}_{\nabla f}g=N(|\nabla f|)\frac {df}{|\nabla f|}\otimes\frac {df}{|\nabla f|} +
b|\nabla f|^2 \frac {db}{df}g'
$$
and 
\bea
 \frac 12 {\mathcal L}_{\nabla f}g &=& Ddf=fz-\frac {n\kappa +sf}{n(n-1)}g\\
 &=& N(|\nabla f|)\frac {df}{|\nabla f|}\otimes\frac {df}{|\nabla f|} +fz -f\alpha \frac {df}{|\nabla f|}\otimes\frac {df}{|\nabla f|}-\frac {n\kappa +sf}{n(n-1)}b^2 g'.
 \eea
By comparing these two equations, we have
\be 
 \left( b|\nabla f|^2 \frac{db}{df}+\frac {n\kappa+sf}{n(n-1)}b^2\right) g'
=f\left( z-\a \frac {df}{|\nabla f|}\otimes\frac {df}{|\nabla f|}\right). \label{eq52}
\ee
Let $\{E_i\}_{i=1}^n$ be a local frame with $E_n=N$. Then, by taking the trace of (\ref{eq52}) on
 $f^{-1}(c)$, 
$$   
b|\nabla f|^2 \frac {db}{df}+\frac {n\kappa +sf}{n(n-1)}b^2  = -\frac{f\a}{n-1},
$$
which implies that 
$$ 
-\frac {\a}{n-1} g' = z-\alpha \frac {df}{|\nabla f|}\otimes\frac {df}{|\nabla f|}.
$$
In particular, on the level hypersurface $f^{-1}(c)$, we have
$$
z(E_i, E_i)=-\frac {\a}{n-1}
$$
for $1\leq i\leq n-1$, and so $T=0$ on $f^{-1}(c)$ by (\ref{eqn2021-3-8-2}). 
The same argument holds for $f^{-1}(t)$ 
instead of $f^{-1}(c)$. Thus we have $T= 0$ and so $C = 0$ on the whole $M$.
Since the scalar curvature is constant, $C=0$ implies $(M, g)$ has harmonic curvature, i.e.,
${\rm div}R = 0$ for the Riemannian curvature tensor $R$ on $(M, g)$. Hence the conclusion follows from the main result in \cite{b-b-b} since $\partial M$ is connected and $(M, g)$ has PIC.
We would like to mention that since $M$ is contractible by Theorem~\ref{str02}, 
the case that a warped product $I \times {\Bbb S}^{n-1}$ with a covering group ${\Bbb Z}_2$
does not happen.

\end{proof}

\end{document}